\documentclass[11pt]{amsart}

\usepackage[T1]{fontenc}
\usepackage[utf8]{inputenc}
\usepackage{lmodern}
\usepackage{microtype}
\usepackage{amsmath,amssymb,amsthm,mathtools}
\usepackage{mathrsfs}
\usepackage{enumitem}
\usepackage{xcolor}
\usepackage{hyperref}
\usepackage[nameinlink,noabbrev]{cleveref}

\hypersetup{
  colorlinks=true,
  linkcolor=blue!50!black,
  citecolor=blue!50!black,
  urlcolor=blue!50!black,
  pdftitle={The Domb Apéry-limit and a proof of the Ramanujan Machine conjecture Z2},
  pdfauthor={Alex Shvets}
}

\newtheorem{theorem}{Theorem}[section]
\newtheorem{proposition}[theorem]{Proposition}
\newtheorem{lemma}[theorem]{Lemma}
\newtheorem{corollary}[theorem]{Corollary}
\theoremstyle{definition}

\theoremstyle{remark}
\newtheorem{remark}[theorem]{Remark}

\newcommand{\Q}{\mathbb{Q}}
\newcommand{\Z}{\mathbb{Z}}
\newcommand{\C}{\mathbb{C}}
\newcommand{\Hh}{\mathfrak{H}}
\newcommand{\dd}{\mathrm{d}}
\newcommand{\e}{\mathrm{e}}
\newcommand{\ii}{\mathrm{i}}
\newcommand{\etaf}{\eta}
\newcommand{\Efour}{E_4}
\newcommand{\sigthree}{\sigma_3}
\newcommand{\Dop}{D}
\newcommand{\thet}{\theta}
\newcommand{\wt}[2]{\left.#1\right|_{#2}W}
\newcommand{\sm}[4]{\begin{psmallmatrix}#1&#2\\#3&#4\end{psmallmatrix}}
\DeclareMathOperator{\PCF}{PCF}
\DeclareMathOperator{\Res}{Res}

\title[The Domb Apéry-limit and RM conjecture Z2]{The Domb Apéry-limit and a proof of the Ramanujan Machine conjecture Z2}
\author{Alex Shvets}
\address{Haifa, Israel}
\email{alex@shvets.io}

\begin{document}

\begin{abstract}
We prove the Domb Apéry-limit
\[
\lim_{n\to\infty}\frac{B_n}{D_n}=\frac{7}{24}\zeta(3),
\]
where $D_n$ are the Domb numbers and $B_n$ is the rational companion sequence satisfying the same three-term recurrence with initial data $B_0=0$, $B_1=1$. As corollaries we obtain the identity
\[
\sum_{n\ge 1}\frac{64^n}{n^3D_nD_{n-1}}=\frac{56}{3}\zeta(3)
\]
and a proof of the Ramanujan Machine continued-fraction conjecture Z2:
\[
\PCF\bigl((2n+1)(5n^2+5n+2),-16n^6\bigr)=\frac{12}{7\zeta(3)}.
\]
The argument combines the level-$6$ eta-product parametrization of the Domb generating function with an Atkin-Lehner transformation law for a weight $-2$ Eichler integral and a Mellin-transform computation of the associated period polynomial.
\end{abstract}

\maketitle

\section{Introduction}

Let
\begin{equation}\label{eq:Domb-def}
D_n:=\sum_{k=0}^n \binom{n}{k}^2\binom{2k}{k}\binom{2(n-k)}{n-k}
\qquad(n\ge 0)
\end{equation}
be the Domb numbers (OEIS~A002895). Their ordinary generating function
\[
\mathcal A(z):=\sum_{n\ge 0}D_nz^n
\]
is the unique holomorphic solution at $z=0$ of the third-order differential equation
\begin{equation}\label{eq:domb-theta-ode}
\bigl[\thet^3-2z(2\thet+1)(5\thet^2+5\thet+2)+64z^2(\thet+1)^3\bigr]y=0,
\qquad \thet=z\frac{\dd}{\dd z}.
\end{equation}
This sequence lies at the center of several Apéry-like phenomena: it admits hypergeometric and modular parametrizations, appears in the theory of random walks and Bessel moments, and it is exactly the sequence used by Cohen in his normalization of the Ramanujan Machine conjecture Z2 for $\zeta(3)$ \cite{BSWZ,Cohen23,Raayoni2021,Yang2008,Zhou1911}.

The Apéry-limit
\begin{equation}\label{eq:main-limit-intro}
\lim_{n\to\infty}\frac{B_n}{D_n}=\frac{7}{24}\zeta(3)
\end{equation}
for the rational companion sequence $B_n$ is listed in Almkvist--van Enckevort--van Straten--Zudilin \cite{AVSZ} as a known value, but that paper gives neither a proof nor a bibliographic pointer for the Domb case. On the other hand, Cohen's Domb sum
\begin{equation}\label{eq:cohen-intro}
\sum_{n\ge 1}\frac{64^n}{n^3D_nD_{n-1}}=\frac{56}{3}\zeta(3)
\end{equation}
appears in \cite{Cohen23} in experimental form. The purpose of the present note is to fill this gap and to deduce from it the Ramanujan Machine continued fraction Z2.

The main results are the following.

\begin{theorem}[Domb Apéry-limit]\label{thm:A}
Let $\{D_n\}_{n\ge 0}$ be given by \eqref{eq:Domb-def}. Let $\{B_n\}_{n\ge 0}\subset\Q$ be defined by
\begin{equation}\label{eq:B-recurrence-intro}
(n+1)^3u_{n+1}=2(2n+1)(5n^2+5n+2)u_n-64n^3u_{n-1}
\qquad(n\ge 1),
\end{equation}
with $B_0=0$ and $B_1=1$. Then
\[
\lim_{n\to\infty}\frac{B_n}{D_n}=\frac{7}{24}\zeta(3).
\]
\end{theorem}

\begin{theorem}[Domb sum]\label{thm:C}
One has
\[
\sum_{n\ge 1}\frac{64^n}{n^3D_nD_{n-1}}=\frac{56}{3}\zeta(3).
\]
\end{theorem}

\begin{theorem}[Ramanujan Machine conjecture Z2]\label{thm:B}
The continued fraction
\[
2+\cfrac{-16\cdot 1^6}{36+\cfrac{-16\cdot 2^6}{160+\cfrac{-16\cdot 3^6}{434+\ddots}}}
\]
converges and its value is
\[
\frac{12}{7\zeta(3)}.
\]
Equivalently,
\[
\PCF\bigl((2n+1)(5n^2+5n+2),-16n^6\bigr)=\frac{12}{7\zeta(3)}.
\]
\end{theorem}

The proof of \cref{thm:A} occupies \S\S\ref{sec:modular}--\ref{sec:proof-A}. Theorems~\ref{thm:C} and \ref{thm:B} then follow by finite-difference and continuant algebra.

\section{Domb recurrence, companion sequence, and continuants}\label{sec:recurrence}

The Domb numbers satisfy the well-known three-term recurrence
\begin{equation}\label{eq:Domb-rec}
(n+1)^3D_{n+1}=2(2n+1)(5n^2+5n+2)D_n-64n^3D_{n-1}
\qquad(n\ge 1),
\end{equation}
with $D_0=1$ and $D_1=4$; equivalently,
\begin{equation}\label{eq:Domb-rec-shifted}
(n+2)^3D_{n+2}-(2n+3)(10n^2+30n+24)D_{n+1}+64(n+1)^3D_n=0
\qquad(n\ge 0).
\end{equation}
See, for instance, \cite[Table~2, case $6(B)$]{Cooper2023}.

Let $\{B_n\}_{n\ge 0}\subset\Q$ be the unique sequence satisfying the same recurrence \eqref{eq:B-recurrence-intro} and the initial values $B_0=0$, $B_1=1$.

\begin{lemma}[discrete Wronskian]\label{lem:wronskian}
Define
\[
W_n:=D_nB_{n-1}-D_{n-1}B_n\qquad(n\ge 1).
\]
Then
\[
W_n=-\frac{64^{n-1}}{n^3}\qquad(n\ge 1).
\]
Consequently,
\begin{equation}\label{eq:increment-ratio}
\frac{B_n}{D_n}-\frac{B_{n-1}}{D_{n-1}}=\frac{64^{n-1}}{n^3D_nD_{n-1}}
\qquad(n\ge 1).
\end{equation}
\end{lemma}

\begin{proof}
From \eqref{eq:B-recurrence-intro} and \eqref{eq:Domb-rec} we obtain
\begin{align*}
W_{n+1}
&=D_{n+1}B_n-D_nB_{n+1}\\
&=\frac{1}{(n+1)^3}\Bigl(2(2n+1)(5n^2+5n+2)D_n-64n^3D_{n-1}\Bigr)B_n\\
&\qquad -\frac{1}{(n+1)^3}D_n\Bigl(2(2n+1)(5n^2+5n+2)B_n-64n^3B_{n-1}\Bigr)\\
&=\frac{64n^3}{(n+1)^3}\,(D_nB_{n-1}-D_{n-1}B_n)
=\frac{64n^3}{(n+1)^3}W_n.
\end{align*}
Since $W_1=D_1B_0-D_0B_1=-1$, induction gives the stated formula. Dividing by $D_nD_{n-1}$ yields \eqref{eq:increment-ratio}.
\end{proof}

\begin{corollary}[finite telescoping identity]\label{cor:telescoping}
For every $N\ge 1$,
\begin{equation}\label{eq:finite-telescoping}
\sum_{n=1}^N\frac{64^n}{n^3D_nD_{n-1}}=64\frac{B_N}{D_N}.
\end{equation}
\end{corollary}

\begin{proof}
Sum \eqref{eq:increment-ratio} from $n=1$ to $n=N$ and use $B_0=0$.
\end{proof}

\begin{proof}[Proof of \cref{thm:C} assuming \cref{thm:A}]
Let $N\to\infty$ in \eqref{eq:finite-telescoping}. By \cref{thm:A} the right-hand side tends to
\[
64\cdot \frac{7}{24}\zeta(3)=\frac{56}{3}\zeta(3).
\]
\end{proof}

We next connect the Domb recurrence to the Ramanujan Machine continued fraction Z2.

\begin{proposition}[continuants and Domb normalization]\label{prop:continuants}
Let
\[
a_0:=2,
\qquad a_n:=(2n+1)(5n^2+5n+2)
\quad(n\ge 1),
\qquad b_n:=-16n^6.
\]
Define continuants by
\[
P_{-1}=1,\quad P_0=2,\qquad Q_{-1}=0,\quad Q_0=1,
\]
\begin{equation}\label{eq:continuant-rec}
U_n=a_nU_{n-1}+b_nU_{n-2}\qquad(n\ge 1),
\end{equation}
where $U_n=P_n$ or $U_n=Q_n$. Then
\begin{equation}\label{eq:continuants-Domb}
P_n=\frac{(n+1)!^3}{2^{n+1}}D_{n+1},
\qquad
Q_n=\frac{(n+1)!^3}{2^n}B_{n+1}
\qquad(n\ge 0).
\end{equation}
Hence
\begin{equation}\label{eq:convergent-ratio}
\frac{P_n}{Q_n}=\frac{D_{n+1}}{2B_{n+1}}.
\end{equation}
\end{proposition}

\begin{proof}
Set
\[
\widetilde P_n:=\frac{(n+1)!^3}{2^{n+1}}D_{n+1}.
\]
Using \eqref{eq:Domb-rec} with index $n$ we obtain
\begin{align*}
\widetilde P_n
&=\frac{(n+1)!^3}{2^{n+1}}D_{n+1}\\
&=\frac{(n+1)!^3}{2^{n+1}}\cdot \frac{2(2n+1)(5n^2+5n+2)D_n-64n^3D_{n-1}}{(n+1)^3}\\
&=(2n+1)(5n^2+5n+2)\frac{n!^3}{2^n}D_n-16n^6\frac{(n-1)!^3}{2^{n-1}}D_{n-1}\\
&=a_n\widetilde P_{n-1}+b_n\widetilde P_{n-2}
\end{align*}
for $n\ge 1$. Also $\widetilde P_0=2$ and $\widetilde P_1=56=a_1\widetilde P_0+b_1P_{-1}$, so $\widetilde P_n=P_n$ by uniqueness for \eqref{eq:continuant-rec}.

The same calculation with $B_n$ in place of $D_n$ gives
\[
\widetilde Q_n:=\frac{(n+1)!^3}{2^n}B_{n+1}
\]
satisfying \eqref{eq:continuant-rec}. Here $\widetilde Q_0=1$ and $\widetilde Q_1=36=a_1$, hence $\widetilde Q_n=Q_n$. Formula \eqref{eq:convergent-ratio} follows immediately.
\end{proof}

\begin{proof}[Proof of \cref{thm:B} assuming \cref{thm:A}]
By construction, $P_n/Q_n$ is the $n$th convergent of the continued fraction in \cref{thm:B}. From \eqref{eq:convergent-ratio} and \cref{thm:A} we get
\[
\lim_{n\to\infty}\frac{P_n}{Q_n}
=\frac{1}{2\lim\limits_{n\to\infty}B_{n+1}/D_{n+1}}
=\frac{1}{2\cdot \frac{7}{24}\zeta(3)}
=\frac{12}{7\zeta(3)}.
\qedhere
\]
\end{proof}

\section{Modular parametrization and the Eichler integral}\label{sec:modular}

We write
\[
\mathcal A(z):=\sum_{n\ge 0}D_nz^n,
\qquad
\mathcal B(z):=\sum_{n\ge 0}B_nz^n.
\]
By construction, $\mathcal A$ satisfies \eqref{eq:domb-theta-ode}, while $\mathcal B$ satisfies the inhomogeneous equation
\begin{equation}\label{eq:B-theta-inhom}
\bigl[\thet^3-2z(2\thet+1)(5\thet^2+5\thet+2)+64z^2(\thet+1)^3\bigr]\mathcal B=z.
\end{equation}
In ordinary differential form these become
\begin{align}
&z^2(4z-1)(16z-1)y''' + 3z(128z^2-30z+1)y'' \notag\\
&\qquad + (448z^2-68z+1)y' +4(16z-1)y =0,
\label{eq:A-ordinary}
\end{align}
for $y=\mathcal A$, and the same equation with right-hand side $z$ for $y=\mathcal B$.

The modular parametrization is the level-$6$ eta-product identity (Chan--Zudilin; Zhou)
\begin{equation}\label{eq:xi-def}
\xi(\tau):=-\frac{\etaf(2\tau)^6\etaf(6\tau)^6}{\etaf(\tau)^6\etaf(3\tau)^6},
\qquad
A(\tau):=\frac{\etaf(\tau)^4\etaf(3\tau)^4}{\etaf(2\tau)^2\etaf(6\tau)^2},
\end{equation}
with
\begin{equation}\label{eq:domb-param}
A(\tau)=\mathcal A(\xi(\tau))=\sum_{n\ge 0}D_n\,\xi(\tau)^n.
\end{equation}
See \cite[(3.2)--(3.7)]{Zhou1911}; compare also \cite{BSWZ}.

The Atkin--Lehner involution relevant here is
\begin{equation}\label{eq:W-def}
W(\tau):=\frac{3\tau-2}{6\tau-3},
\qquad
W=\sm{3}{-2}{6}{-3},
\qquad
\det W=3,
\qquad
W^2=-3I.
\end{equation}
We use the usual slash operator for matrices of positive determinant:
\[
(f|_k M)(\tau):=(\det M)^{k/2}(c\tau+d)^{-k}f\!\left(\frac{a\tau+b}{c\tau+d}\right),
\qquad M=\sm{a}{b}{c}{d}.
\]
From \cite[(3.4), (3.7)]{Zhou1911} we have
\begin{equation}\label{eq:xi-A-W}
\xi(W\tau)=\xi(\tau),
\qquad
A(W\tau)=-\frac{(6\tau-3)^2}{3}A(\tau),
\end{equation}
that is,
\begin{equation}\label{eq:A-slash-W}
A|_2W=-A.
\end{equation}

Define
\begin{equation}\label{eq:E-def}
E(\tau):=\frac{\mathcal B(\xi(\tau))}{A(\tau)}.
\end{equation}
The next proposition identifies $E$ with a weight $-2$ Eichler integral.

\begin{proposition}\label{prop:Eichler}
Let
\begin{equation}\label{eq:g-def}
g(\tau):=\frac{\Efour(\tau)-\Efour(2\tau)-9\Efour(3\tau)+9\Efour(6\tau)}{240}
=\sum_{n\ge 1}a_nq^n,
\qquad q=\e^{2\pi\ii\tau},
\end{equation}
where
\begin{equation}\label{eq:a-n-def}
a_n=\sigthree(n)-\sigthree(n/2)-9\sigthree(n/3)+9\sigthree(n/6).
\end{equation}
Then
\begin{equation}\label{eq:E-qseries}
E(\tau)=-\sum_{n\ge 1}\frac{a_n}{n^3}q^n,
\qquad\text{equivalently}\qquad
\Dop^3E=-g,
\end{equation}
where $\Dop=q\,\dd/\dd q=(2\pi\ii)^{-1}\dd/\dd\tau$.
\end{proposition}

\begin{proof}
Apply the third-order variation-of-constants construction of Yang \cite[Lemmas~1--2]{Yang2008} to the ordinary inhomogeneous equation \eqref{eq:A-ordinary} for $\mathcal B$. Since the leading coefficient in \eqref{eq:A-ordinary} is $z^2(4z-1)(16z-1)$, this gives
\begin{equation}\label{eq:phi-def}
\Dop^3E(\tau)=\Phi(\tau):=\left(\frac{\Dop\xi(\tau)}{\xi(\tau)}\right)^3\frac{\xi(\tau)}{A(\tau)(1-4\xi(\tau))(1-16\xi(\tau))}.
\end{equation}
Because $(\Dop\xi)/\xi$ has weight $2$, $A$ has weight $2$, and $\xi$ is modular of weight $0$, the function $\Phi$ is a modular form of weight $4$ on $\Gamma_0(6)$.

A direct $q$-expansion from \eqref{eq:xi-def} gives
\[
\Phi(\tau)=-q-8q^2-19q^3-64q^4+O(q^5),
\]
while from \eqref{eq:g-def} one has
\[
-g(\tau)=-q-8q^2-19q^3-64q^4+O(q^5).
\]
Thus $\Phi+g\in M_4(\Gamma_0(6))$ has vanishing Fourier coefficients through $q^4$. Since $[\mathrm{SL}_2(\Z):\Gamma_0(6)]=12$, the Sturm bound for weight $4$ and level $6$ is $\lfloor 4\cdot 12/12\rfloor=4$; hence, by Sturm's theorem \cite[Theorem~3.13]{DiamondShurman}, $\Phi=-g$. This proves \eqref{eq:E-qseries}.
\end{proof}

We record the $L$-function of $g$ and its alternating twist.

\begin{proposition}\label{prop:Lfunctions}
The Dirichlet series of $g$ is
\begin{equation}\label{eq:Lg-factor}
L(g,s):=\sum_{n\ge 1}\frac{a_n}{n^s}=\zeta(s)\zeta(s-3)(1-2^{-s})(1-3^{2-s}).
\end{equation}
Hence
\begin{equation}\label{eq:Lg3}
L(g,3)=-\frac{7}{24}\zeta(3).
\end{equation}
Define the alternating twist
\begin{equation}\label{eq:Lstar-def}
L^*(s):=\sum_{n\ge 1}\frac{(-1)^n a_n}{n^s}.
\end{equation}
Then
\begin{equation}\label{eq:Lstar-factor}
L^*(s)=-\zeta(s)\zeta(s-3)(1-2^{-s})(1-2^{4-s})(1-3^{2-s}).
\end{equation}
In particular,
\begin{equation}\label{eq:Lstar-specials}
L^*(3)=-\frac{7}{24}\zeta(3),
\qquad
L^*(2)=0,
\qquad
L^*(1)=\frac{7}{4\pi^2}\zeta(3).
\end{equation}
Moreover, for every $s\in\C$,
\begin{equation}\label{eq:Lstar-fe}
\Gamma(s)\left(\frac{\pi}{\sqrt 3}\right)^{-s}L^*(s+3)
=
\Gamma(-s-2)\left(\frac{\pi}{\sqrt 3}\right)^{s+2}L^*(1-s).
\end{equation}
\end{proposition}

\begin{proof}
Equation \eqref{eq:Lg-factor} follows immediately from \eqref{eq:a-n-def} and the classical identity
\[
\sum_{n\ge 1}\frac{\sigthree(n)}{n^s}=\zeta(s)\zeta(s-3).
\]
Substituting $s=3$ yields \eqref{eq:Lg3} because $\zeta(0)=-1/2$.

For the twist, write $n=2^rm$ with $m$ odd. Since \eqref{eq:a-n-def} implies
\[
a_{2^rm}=8^r a_m\qquad(m\text{ odd},\ r\ge 0),
\]
we get
\begin{align*}
L^*(s)
&=\sum_{m\text{ odd}}\frac{a_m}{m^s}\left(-1+\sum_{r\ge 1}\frac{8^r}{2^{rs}}\right)\\
&=\sum_{m\text{ odd}}\frac{a_m}{m^s}\left(-1+\frac{2^{3-s}}{1-2^{3-s}}\right)
=-\frac{1-2^{4-s}}{1-2^{3-s}}\sum_{m\text{ odd}}\frac{a_m}{m^s}.
\end{align*}
On the other hand, removing the $2$-Euler factor from \eqref{eq:Lg-factor} gives
\[
\sum_{m\text{ odd}}\frac{a_m}{m^s}=(1-2^{3-s})L(g,s).
\]
Combining these formulas with \eqref{eq:Lg-factor} proves \eqref{eq:Lstar-factor}. The values in \eqref{eq:Lstar-specials} follow immediately, using the standard identity
\[
\zeta'(-2)=-\frac{\zeta(3)}{4\pi^2}.
\]

Finally, \eqref{eq:Lstar-fe} is a direct consequence of \eqref{eq:Lstar-factor} and the functional equation of the Riemann zeta function. Indeed,
\begin{align*}
&\Gamma(s)\left(\frac{\pi}{\sqrt 3}\right)^{-s}L^*(s+3)\\
&\qquad =-\Gamma(s)\left(\frac{\pi}{\sqrt 3}\right)^{-s}\zeta(s+3)\zeta(s)(1-2^{-s-3})(1-2^{1-s})(1-3^{-s-1})
\end{align*}
and the two zeta functional equations transform the right-hand side into
\[
\Gamma(-s-2)\left(\frac{\pi}{\sqrt 3}\right)^{s+2}L^*(1-s).
\qedhere
\]
\end{proof}

\section{Local analysis at the dominant singularity}\label{sec:local}

Let
\begin{equation}\label{eq:tau-star}
\tau_*:=\frac12+\frac{\ii}{2\sqrt 3}.
\end{equation}
Then $\tau_*$ is the unique point of the upper half-plane satisfying
\begin{equation}\label{eq:tau-star-quad}
3\tau^2-3\tau+1=0,
\end{equation}
and $W(\tau_*)=\tau_*$. By \cite[(3.11)]{Zhou1911},
\begin{equation}\label{eq:xi-at-taustar}
\xi(\tau_*)=\frac{1}{16}.
\end{equation}
Thus the dominant singularity of $\mathcal A(z)$ corresponds to the order-$2$ elliptic fixed point $\tau_*$. Write
\[
z_0:=\frac{1}{16},
\qquad
\varepsilon:=1-16z.
\]

\begin{proposition}\label{prop:local-exponents}
The point $z_0=1/16$ is a regular singular point of \eqref{eq:A-ordinary}, and its local exponents are
\[
0,\qquad \frac12,\qquad 1.
\]
Consequently, near $z_0$ one has expansions
\begin{equation}\label{eq:A-B-local}
\mathcal A(z)=\alpha_0+\alpha_1\varepsilon^{1/2}+O(\varepsilon),
\qquad
\mathcal B(z)=\beta_0+\beta_1\varepsilon^{1/2}+O(\varepsilon)
\end{equation}
for suitable constants $\alpha_0,\alpha_1,\beta_0,\beta_1\in\C$.
\end{proposition}

\begin{proof}
In \eqref{eq:A-ordinary} the coefficient of $y'''$ has a simple zero at $z_0$, so $z_0$ is a regular singular point. Substitute $z=(1-\varepsilon)/16$ and $y=\varepsilon^r$. Since $\dd/\dd z=-16\,\dd/\dd\varepsilon$, the coefficient of $\varepsilon^{r-2}$ in \eqref{eq:A-ordinary} is
\[
-6r(2r-1)(r-1).
\]
Hence the indicial equation is $r(2r-1)(r-1)=0$, and the exponents are exactly $0$, $1/2$, and $1$.

Now $A(\tau)=\mathcal A(\xi(\tau))$ and $\mathcal B(\xi(\tau))$ are holomorphic in $\tau$ near $\tau_*$. Since $\tau_*$ is an elliptic point of order $2$ for the genus-zero group generated by $\Gamma_0(6)$ and $W$, the Hauptmodul $\xi$ has ramification index $2$ at $\tau_*$; equivalently,
\begin{equation}\label{eq:ramification}
\xi(\tau)-\frac{1}{16}=c(\tau-\tau_*)^2+O\bigl((\tau-\tau_*)^3\bigr)
\qquad(c\ne 0).
\end{equation}
Therefore $\varepsilon^{1/2}$ is a holomorphic local parameter in $\tau$, and the expansions \eqref{eq:A-B-local} follow.
\end{proof}

\begin{lemma}[transfer to coefficients]\label{lem:transfer}
The coefficients of $\mathcal A$ and $\mathcal B$ satisfy
\[
D_n\sim -\frac{\alpha_1}{2\sqrt\pi}\,16^n n^{-3/2},
\qquad
B_n\sim -\frac{\beta_1}{2\sqrt\pi}\,16^n n^{-3/2},
\]
and hence
\begin{equation}\label{eq:limit-branching}
\lim_{n\to\infty}\frac{B_n}{D_n}=\frac{\beta_1}{\alpha_1}.
\end{equation}
\end{lemma}

\begin{proof}
By \cref{prop:local-exponents}, both functions admit local expansions of the form
\[
F(z)=c_0+c_1\bigl(1-z/z_0\bigr)^{1/2}+O(1-z/z_0)
\qquad(z_0=1/16).
\]
Since there is no other singularity on $|z|=z_0$, the Flajolet--Odlyzko transfer theorem applies; see \cite[Theorem~VI.1]{FlajoletSedgewick}. It gives
\[
[z^n]F(z)\sim -\frac{c_1}{2\sqrt\pi}\,z_0^{-n}n^{-3/2}.
\]
Applying this to $\mathcal A$ and $\mathcal B$ yields the stated asymptotics and \eqref{eq:limit-branching}.
\end{proof}

\begin{lemma}\label{lem:derivative-ratio}
One has
\begin{equation}\label{eq:branch-derivative-ratio}
\frac{\beta_1}{\alpha_1}=\frac{\dfrac{\dd}{\dd\tau}\mathcal B(\xi(\tau))\big|_{\tau=\tau_*}}{\dfrac{\dd}{\dd\tau}A(\tau)\big|_{\tau=\tau_*}}.
\end{equation}
\end{lemma}

\begin{proof}
Let $u=\tau-\tau_*$. By \eqref{eq:ramification} we have
\[
\varepsilon^{1/2}=\lambda u+O(u^2)
\qquad(\lambda\ne 0).
\]
Hence \eqref{eq:A-B-local} becomes
\[
A(\tau)=\alpha_0+\alpha_1\lambda u+O(u^2),
\qquad
\mathcal B(\xi(\tau))=\beta_0+\beta_1\lambda u+O(u^2).
\]
Differentiating at $u=0$ gives
\[
A'(\tau_*)=\alpha_1\lambda,
\qquad
\frac{\dd}{\dd\tau}\mathcal B(\xi(\tau))\Big|_{\tau=\tau_*}=\beta_1\lambda,
\]
which proves \eqref{eq:branch-derivative-ratio}.
\end{proof}

\section{Atkin--Lehner transformation law}\label{sec:AL}

We first compute the action of $W$ on the Eisenstein combination $g$.

\begin{lemma}\label{lem:E4-transform}
With $W$ as in \eqref{eq:W-def}, one has
\[
\Efour(\tau)|_4W=9\Efour(3\tau),
\qquad
\Efour(2\tau)|_4W=9\Efour(6\tau),
\]
\[
\Efour(3\tau)|_4W=\frac{1}{9}\Efour(\tau),
\qquad
\Efour(6\tau)|_4W=\frac{1}{9}\Efour(2\tau).
\]
Consequently,
\begin{equation}\label{eq:g-anti}
g|_4W=-g.
\end{equation}
\end{lemma}

\begin{proof}
Factor
\[
W=\sm{3}{-2}{6}{-3}=\sm{1}{-2}{2}{-3}\sm{3}{0}{0}{1}.
\]
Since $\sm{1}{-2}{2}{-3}\in \mathrm{SL}_2(\Z)$ and $\Efour$ is modular on $\mathrm{SL}_2(\Z)$,
\[
\Efour|_4W=(\Efour|_4\sm{1}{-2}{2}{-3})\Big|_4\sm{3}{0}{0}{1}
=\Efour\Big|_4\sm{3}{0}{0}{1}=9\Efour(3\tau).
\]
Similarly,
\[
\sm{2}{0}{0}{1}W=\sm{6}{-4}{6}{-3}=\sm{1}{-4}{1}{-3}\sm{6}{0}{0}{1},
\]
so
\[
\Efour(2\tau)|_4W
=2^{-2}\Efour\Big|_4\sm{2}{0}{0}{1}W
=2^{-2}\Efour\Big|_4\sm{1}{-4}{1}{-3}\sm{6}{0}{0}{1}
=9\Efour(6\tau).
\]
Now $W^2=-3I$, and for even weights $k$ one has $f|_k(-3I)=f$, so $|_kW$ is an involution. Applying $|_4W$ to the first two formulas gives the other two. Substituting these identities into \eqref{eq:g-def} yields \eqref{eq:g-anti}.
\end{proof}

The next lemma is the special case of Bol's identity that we need.

\begin{lemma}\label{lem:bol}
For every holomorphic function $F$ on $\Hh$,
\begin{equation}\label{eq:bol-special}
\frac{\dd^3}{\dd\tau^3}\left(\frac{(6\tau-3)^2}{3}F(W\tau)\right)
=\frac{9}{(6\tau-3)^4}F^{(3)}(W\tau).
\end{equation}
Equivalently,
\[
\Dop^3(\wt{F}{-2})=\wt{\Dop^3F}{4}.
\]
\end{lemma}

\begin{proof}
A direct differentiation suffices. Since
\[
W'(\tau)=\frac{3}{(6\tau-3)^2},
\qquad
W''(\tau)=-\frac{36}{(6\tau-3)^3},
\qquad
W'''(\tau)=\frac{648}{(6\tau-3)^4},
\]
repeated application of the chain rule shows that all terms involving $F$,$F'$,$F''$ cancel in the third derivative, leaving exactly \eqref{eq:bol-special}. Dividing by $(2\pi\ii)^3$ gives the $\Dop$-version.
\end{proof}

We now determine the period polynomial of $E$.

\begin{proposition}[Atkin--Lehner law for $E$]\label{prop:E-transform}
One has
\begin{equation}\label{eq:E-transform}
(\wt{E}{-2})(\tau)+E(\tau)=\frac{7}{6}\zeta(3)\,(3\tau^2-3\tau+1)
=-4L(g,3)\,(3\tau^2-3\tau+1).
\end{equation}
\end{proposition}

\begin{proof}
Set
\[
H(\tau):=(\wt{E}{-2})(\tau)+E(\tau).
\]
By \cref{lem:bol,lem:E4-transform,prop:Eichler},
\[
\Dop^3H=(\wt{\Dop^3E}{4})+\Dop^3E=-(g|_4W)-g=0.
\]
Therefore $H$ is a polynomial of degree at most $2$.

To determine it, restrict to the $W$-invariant geodesic
\begin{equation}\label{eq:tau-y}
\tau(y):=\frac12+\frac{\ii y}{2\sqrt 3}
\qquad(y>0).
\end{equation}
Then $W\tau(y)=\tau(1/y)$ and, by \eqref{eq:E-qseries},
\begin{equation}\label{eq:F-y-def}
F(y):=E(\tau(y))=-\sum_{n\ge 1}\frac{(-1)^na_n}{n^3}\,\exp\!\left(-\frac{\pi n y}{\sqrt 3}\right).
\end{equation}
Mellin inversion gives, for $c>1$,
\begin{equation}\label{eq:Mellin-E}
F(y)=-\frac{1}{2\pi\ii}\int_{(c)}\Gamma(s)\left(\frac{\pi}{\sqrt 3}\right)^{-s}L^*(s+3)y^{-s}\,\dd s.
\end{equation}
By \eqref{eq:Lstar-fe}, the integrand is invariant under $s\mapsto -s-2$. Using this symmetry, shifting the contour, and collecting residues at $s=0,-2$ (there is no residue at $s=-1$ because $L^*(2)=0$) yields
\begin{equation}\label{eq:line-functional}
F(y)-y^2F(1/y)=\frac{7}{24}\zeta(3)(1-y^2).
\end{equation}
Indeed,
\[
\Res_{s=0}=L^*(3)=-\frac{7}{24}\zeta(3),
\]
while
\[
\Res_{s=-2}=\frac{1}{2}\left(\frac{\pi}{\sqrt 3}\right)^2L^*(1)y^2
=\frac{7}{24}\zeta(3)y^2.
\]
Now, on the geodesic \eqref{eq:tau-y},
\[
(\wt{E}{-2})(\tau(y))
=\frac{(6\tau(y)-3)^2}{3}E(W\tau(y))
=-y^2F(1/y),
\]
so \eqref{eq:line-functional} is exactly
\[
H(\tau(y))=\frac{7}{24}\zeta(3)(1-y^2).
\]
Finally,
\[
3\tau(y)^2-3\tau(y)+1=\frac{1-y^2}{4},
\]
whence
\[
H(\tau(y))=\frac{7}{6}\zeta(3)\,(3\tau(y)^2-3\tau(y)+1).
\]
Since both sides are polynomials of degree at most $2$ in $\tau$ and they agree on the infinite set $\{\tau(y):y>0\}$, identity \eqref{eq:E-transform} follows.
\end{proof}

\section{Proof of the Domb Apéry-limit}\label{sec:proof-A}

We first differentiate the transformation law at the fixed point $\tau_*$. Since
\[
3\tau_*^2-3\tau_*+1=0,
\qquad 6\tau_*-3=\ii\sqrt 3,
\qquad W'(\tau_*)=-1,
\]
differentiating \eqref{eq:E-transform} at $\tau=\tau_*$ gives
\begin{equation}\label{eq:E-derivative-relation}
E(\tau_*)+\frac{E'(\tau_*)}{2\ii\sqrt 3}=\frac{7}{24}\zeta(3).
\end{equation}
Indeed, differentiating the left-hand side of \eqref{eq:E-transform} yields
\[
\frac{\dd}{\dd\tau}(\wt{E}{-2})(\tau_*)+E'(\tau_*)=4\ii\sqrt 3\,E(\tau_*)+E'(\tau_*)+E'(\tau_*),
\]
which equals $4\ii\sqrt3\,E(\tau_*)+2E'(\tau_*)$; differentiating the right-hand side gives
\[
\frac{7}{6}\zeta(3)(6\tau_*-3)=\frac{7}{6}\zeta(3)\,\ii\sqrt 3.
\]
After division by $4\ii\sqrt 3$ we obtain \eqref{eq:E-derivative-relation}.

The same differentiation applied to \eqref{eq:A-slash-W} yields the derivative of $A$.

\begin{lemma}\label{lem:A-derivative}
At the fixed point $\tau_*$ one has
\begin{equation}\label{eq:A-derivative}
A'(\tau_*)=2\ii\sqrt 3\,A(\tau_*).
\end{equation}
\end{lemma}

\begin{proof}
Equation \eqref{eq:A-slash-W} is equivalent to
\[
A(W\tau)=-\frac{(6\tau-3)^2}{3}A(\tau).
\]
Differentiate and evaluate at $\tau_*$. Since $W(\tau_*)=\tau_*$ and $W'(\tau_*)=-1$, we obtain
\[
-A'(\tau_*)=-4(6\tau_*-3)A(\tau_*)-\frac{(6\tau_*-3)^2}{3}A'(\tau_*).
\]
Because $(6\tau_*-3)^2=-3$, the last term equals $+A'(\tau_*)$, and the identity simplifies to
\[
2A'(\tau_*)=4(6\tau_*-3)A(\tau_*)=4\ii\sqrt3\,A(\tau_*),
\]
which is \eqref{eq:A-derivative}.
\end{proof}

We now complete the proof.

\begin{proof}[Proof of \cref{thm:A}]
From \eqref{eq:E-def} we have
\[
\mathcal B(\xi(\tau))=A(\tau)E(\tau).
\]
Differentiating at $\tau_*$ and using \cref{lem:A-derivative} gives
\begin{align*}
\frac{\dfrac{\dd}{\dd\tau}\mathcal B(\xi(\tau))\big|_{\tau=\tau_*}}{A'(\tau_*)}
&=\frac{A'(\tau_*)E(\tau_*)+A(\tau_*)E'(\tau_*)}{A'(\tau_*)}\\
&=E(\tau_*)+\frac{E'(\tau_*)}{2\ii\sqrt 3}.
\end{align*}
By \eqref{eq:E-derivative-relation} the right-hand side equals $\frac{7}{24}\zeta(3)$. Therefore, by \cref{lem:derivative-ratio},
\[
\frac{\beta_1}{\alpha_1}=\frac{7}{24}\zeta(3).
\]
Finally, \cref{lem:transfer} gives
\[
\lim_{n\to\infty}\frac{B_n}{D_n}=\frac{\beta_1}{\alpha_1}=\frac{7}{24}\zeta(3),
\]
as claimed.
\end{proof}

\begin{remark}
The proof shows that the relevant constant is \emph{not} the naive CM-value $E(\tau_*)$. Instead, it is the linear coefficient of the Eichler integral at the order-$2$ elliptic point, namely
\[
E(\tau_*)+\frac{E'(\tau_*)}{2\ii\sqrt3}=\frac{7}{24}\zeta(3).
\]
This is the correct replacement for the false identity $E(\tau_*)=-L(g,3)$.
\end{remark}

\end{document}